\documentclass[a4paper,11pt, twoside]{amsart}
\pdfoutput=1

\usepackage{dsfont}
\usepackage{multicol}	     % used for the two-column index
\usepackage[bottom]{footmisc}% places footnotes at page bottom

\usepackage[title]{appendix}

\usepackage{pgf, tikz, tikz-cd}
\usepackage{dsfont}
\usetikzlibrary{matrix}
\usetikzlibrary{arrows}
\usetikzlibrary{patterns}
\usepackage{mathtools}
\mathtoolsset{showonlyrefs=true}
\usepackage{amsmath,amssymb,amsthm,amsfonts, mathrsfs, fancyhdr}

\usepackage[leqno]{amsmath}
\makeatletter
\newcommand{\leqnomode}{\tagsleft@true\let\veqno\@@leqno}
\newcommand{\reqnomode}{\tagsleft@false\let\veqno\@@eqno}
\makeatother

\usepackage{anyfontsize}

\usepackage{upgreek}

\usepackage{enumitem}
\makeatletter
\newcommand{\mylabel}[2]{#2\def\@currentlabel{#2}\label{#1}}
\makeatother

\usepackage{todonotes}

\setlist[description]{leftmargin=*}
\usepackage{amscd}
\usepackage[active]{srcltx}
\usepackage{verbatim}
\usepackage{epsfig,graphicx,color}
\usepackage{graphicx}
\usepackage{mathtools,xparse}
\usepackage[english]{babel}

\usepackage{tabularx,ragged2e}
\newcolumntype{L}{>{\RaggedRight\arraybackslash}X}

\usepackage{xltabular}
\usepackage{booktabs}
\usepackage{caption}
\usepackage{subcaption}

\usepackage{pgf, tikz, tikz-cd}
\usepackage{dsfont}
\usetikzlibrary{matrix}

\usepackage{capt-of}
\usepackage{float}

\usepackage{verbatim}
\usepackage{cite}
\usepackage{manyfoot}

\usepackage{fouriernc}
\usepackage[left=2.7cm,right=2.7cm,top=2.3cm,bottom=2.3cm]{geometry}

\usepackage[unicode,bookmarks, pdftex, backref = page]{hyperref}
\definecolor{newblue}{RGB}{0,102,204}
\definecolor{newred}{RGB}{206,32,41}

\hypersetup{colorlinks=true,citecolor=newblue,linkcolor=newred,
  urlcolor=magenta,
pdfpagemode=UseNone, breaklinks=true}

\usepackage{apptools}
\AtAppendix{\counterwithin{theorem}{section}}

\newtheorem{theorem}{Theorem}[section]
\newtheorem{proposition}[theorem]{Proposition}

\newtheorem{corollary}[theorem]{Corollary}
\newtheorem{question}[theorem]{Question}

\newtheorem*{maintheorem}{Main Theorem}

\theoremstyle{definition}
\newtheorem{remark}[theorem]{Remark}
\newtheorem{example}[theorem]{Example}
\newtheorem{definition}[theorem]{Definition}

\renewcommand{\setminus}{-}

\newcommand{\mZ}{\mathbb{Z}}

\renewcommand{\to}{\longrightarrow}
\renewcommand{\t}{\mathsf{t}}
\renewcommand{\r}{\mathsf{r}}
\newcommand{\z}{\mathsf{z}}

\renewcommand{\S}{\mathfrak{S}}
\renewcommand{\P}{\mathfrak{P}}

\usepackage{epigraph}
\usepackage{listings}
\usepackage{color}

\lstdefinestyle{customc}{
  belowcaptionskip=1\baselineskip,
  breaklines=true,
 frame=L,
  xleftmargin=\parindent,
  language=GAP,
  showstringspaces=false,
  basicstyle=\tiny\ttfamily,
  keywordstyle=\bfseries\color{green!40!black},
  commentstyle=\itshape\color{purple!40!black},
  identifierstyle=\color{blue},
  stringstyle=\color{olive},
}

\definecolor{codegreen}{rgb}{0,0.6,0}
\definecolor{codegray}{rgb}{0.5,0.5,0.5}
\definecolor{codepurple}{rgb}{0.58,0,0.82}
\definecolor{backcolour}{rgb}{0.95,0.95,0.92}

\lstdefinestyle{mystyle}{
    backgroundcolor=\color{backcolour},   
    commentstyle=\color{codegreen},
    keywordstyle=\color{magenta},
    numberstyle=\tiny\color{codegray},
    stringstyle=\color{codepurple},
    basicstyle=\ttfamily\tiny,
    breakatwhitespace=false,         
    breaklines=true,                 
    captionpos=t,                    
    keepspaces=true,                 
    numbers=left,                    
    numbersep=5pt,                  
    showspaces=false,                
    showstringspaces=false,
    showtabs=false,                  
    tabsize=2
}

\title[On finite quotients of pure surface braid groups]{On finite quotients of pure surface braid groups \\ having order at most $127$}

\date{}
\author[Francesco Polizzi]{Francesco Polizzi $^{*}$}
\address{Dipartimento di Matematica e Applicazioni ``Renato Caccioppoli''
  \newline\indent
  Universit\`a degli Studi di Napoli Federico II
  \newline\indent
Via Cintia, Monte S. Angelo
  \newline\indent
  I-80126  Napoli, Italy}
\email{francesco.polizzi@unina.it}

\author[Pietro Sabatino]{Pietro Sabatino}
\address{Institute for High Performance Computing and Networking (ICAR-CNR)
    \newline\indent
via P. Bucci 8/9C
\newline\indent
87036 Rende (CS), Italy}
\email{pietro.sabatino@icar.cnr.it}

\thanks{\emph{2010 Mathematics Subject
Classification.} 20F36, 20D15, 05E16}

\keywords{Surface braid groups, Finite groups, Complex surfaces}

\begin{document}

%\epigraph{\itshape Here we insert a nice citation, if
%we %find one.}{Author, \textit{citation}}

\begin{abstract}
Let $\Sigma_b$ be a compact Riemann surface of genus $b \geq 2$ and let $\mathsf{P}_2(\Sigma_b)=\pi_1(\Sigma_b \times \Sigma_b - \Delta)$ be the corresponding pure braid group on two strands. A finite quotient $G$ of $\mathsf{P}_2(\Sigma_b)$ is called \emph{admissible} if the quotient homomorphism $\varphi \colon \mathsf{P}_2(\Sigma_b) \to G$  does not factor through $\pi_1(\Sigma_b \times \Sigma_b)$. In this article we classify all admissible quotients of $\mathsf{P}_2(\Sigma_b)$ such that $|G| \leq 127$, also reviewing some of our previous work on this subject.
\end{abstract}

\maketitle

%\tableofcontents

%\epigraph{\itshape Here we insert a nice citation, if we find one.}
%{Author, \textit{citation}}

%\setcounter{section}{-1}

\section{Introduction} \label{sec:intro}
Let $\Sigma_b$ be a compact Riemann surface of genus $b$. We write $\mathsf{P}_2(\Sigma_b)$ for the pure braid group on two strands of $\Sigma_b$, defined as the fundamental group $\pi_1(\Sigma_b \times \Sigma_b - \Delta)$, where $\Delta \subset \Sigma_b \times \Sigma_b$ is the diagonal. Surface braid groups are finitely presented for every choice of $b$, and several presentations appear in the literature, see for instance \cite{Bir69, Scott70, Bel04, GG04}; they play a relevant role in low-dimensional topology and related fields, bridging ideas from algebraic topology, geometric topology, and algebraic geometry.

Surface braid groups $\mathsf{P}_2(\Sigma_b)$ are residually $p$-finite for all primes $p$, see \cite[pp. 1481--1490]{BarBel09}. The present work, together with our previous papers on this topic, aims at an explicit classification of some of their finite quotients, motivated by geometric considerations.

Let $A_{12}$ denote the homotopy class of a geometric braid that loops once around the diagonal. A surjective homomorphism 
\begin{equation} \label{eq:admissible-epi}   
\varphi \colon \mathsf{P}_2(\Sigma_b) \to G
\end{equation}
onto a finite group $G$, such that $\varphi(A_{12})$ has order $n \geq 2$, will be called an \emph{admissible epimorphism}. We will also say that the finite group $G$ is an \emph{admissible} (\emph{pure}) \emph{braid quotient of tipe} $(b, \, n)$. Equivalently, $\varphi$ is admissible if it does not factor through $\pi_1(\Sigma_b \times \Sigma_b)$. 

In the sequel, we always suppose $b \geq 2$. Under this assumption, a series of recent works \cite{CaPol19, Pol22, PolSab22, PolSab23, PolSab24} has clarified the role of admissible quotients of $\mathsf{P}_2(\Sigma_b)$ in algebraic geometry, and in particular in the theory of double Kodaira surfaces. We refer the reader to these papers for further details, and to \cite{Rol10, Cat17, LLR17} for general background on double Kodaira surfaces. 

In the case $b=2$, admissible quotients of order at most $32$ were classified in \cite{PolSab22}, and this classification was subsequently extended up to order $64$ in \cite{PolSab24}. In the present paper, which may be viewed both as a review and as a continuation of these works, we provide a complete classification of admissible quotients with $|G| \leq 127$, for every genus $b \geq 2$. Our main result is the following, see Theorem \ref{thm:main}.
\begin{maintheorem} \label{thm:A}
Let $G$ be an admissible braid quotient of type $(b, \, n)$. If $|G| \leq 127$, then necessarily  $(b, \, n)=(2, \, 2)$, and all the occurrences are listed in Table \emph{\ref{tab:admissible_127_intro}} below.
\begin{table}[H]  
    \scriptsize
%\caption{Groups of order $\leq 64$ admitting prestructures}  
    \centering
    \begin{tabularx}{0.75\linewidth}{@{}cccc@{}} 
\toprule
& &  $\mathrm{Number} \, \, \mathrm{of}$ &   $\mathrm{Total} \, \, \mathrm{number} \, \, \mathrm{of}$ \\
&&  $\mathrm{admissible} \, \, \mathrm{epimorphisms}$ &   $\mathrm{admissible} \, \, \mathrm{epimorphisms}$  \\

 &   &  $\varphi \colon \mathsf{P}_2(\Sigma_2) \to G $  &  $\varphi \colon \mathsf{P}_2(\Sigma_2) \to G $ \\
      $\mathrm{IdSmallGroup}(G)$ & $\mathrm{Is} \, \, G\, \,  \, \mathrm{monolithic}?$ &  
      $ \mathrm{up} \, \, \mathrm{ to} \, \, \operatorname{Aut}(G)$ & \\
      \toprule
$G(32, \, 49)$ & $\mathrm{Yes}$  & $1920$  & $2211840$\\    
$G(32, \, 50)$ & $\mathrm{Yes}$  & $1152$ & $2211840$ \\       
$G(64, \, 199)$ & $\mathrm{No}$  & $138240$ & $566231040$ \\  
$G(64, \, 200)$ & $\mathrm{No}$  & $46080$ & $566231040$\\ 
$G(64, \, 201)$ & $\mathrm{No}$ & $184320$ & $566231040$ \\ 
$G(64, \, 249)$ & $\mathrm{Yes}$  & $368640$ & $566231040$ \\ 
$G(64, \, 264)$ & $\mathrm{No}$  & $14400$ &  $530841600$\\ 
$G(64, \, 265)$ & $\mathrm{No}$  & $8640$ & $530841600$\\ 
$G(64, \, 266)$ & $\mathrm{Yes}$  & $23040$ & $530841600$\\ 
$G(96, \, 224)$ & $\mathrm{No}$ & $6297600$ & $14509670400$\\
$G(96, \, 225)$ & $\mathrm{No}$ & $3778560$ &  $14509670400$ \\
\bottomrule
\end{tabularx} 
\caption{Admissible braid quotients of order at most $127$}
\label{tab:admissible_127_intro} 
 \end{table} 
\end{maintheorem}
Our Main Theorem shows that admissible braid quotients are extremely rare: among the $1036$ non-abelian groups of order at most $127$, only $11$ are admissible braid quotients.\footnote{Without the admissibility condition, many more quotients occur, since $\pi_1(\Sigma_b \times \Sigma_b)=\pi_1(\Sigma_b)\times \pi_1(\Sigma_b)$; for instance, all $4b$-generated abelian groups arise.}
Although the cases of orders $32$ and $64$ were already studied in previous work as admissible quotients of $\mathsf{P}_2(\Sigma_2)$, the present paper shows that none of these groups occurs as an admissible quotient of $\mathsf{P}_2(\Sigma_b)$ for any $b\ge 3$. Thus, every admissible quotient of $\mathsf{P}_2(\Sigma_b)$ with $b\ge 3$ satisfies $|G|\ge 128$ (Corollary~\ref{cor:b=3_128}). On the other hand, it is known from~\cite{CaPol19} that the two extraspecial groups of order $2^7=128$ are admissible quotients of $\mathsf{P}_2(\Sigma_3)$, showing that this bound is sharp. We also note that the two admissible quotients of order $96$ appear here for the first time.

The proof of the Main Theorem combines methods from combinatorial group theory with explicit computations performed using the computer algebra system \verb|GAP4|~\cite{GAP4}. To this end, we implemented a dedicated procedure, \verb|CheckStructures|. All relevant scripts are available in a public GitHub repository~\cite{GAP4DDKS}, and selected excerpts of the \verb|GAP4| code are included in the paper to illustrate the computational methods involved.

\bigskip

$\mathbf{Notation \; and \; conventions}$.
The order of a finite group $G$ is denoted by $|G|$. If $x
    \in G$, the order of
    $x$ is denoted by $o(x)$. If $S= \{s_1, \ldots, s_n \} \subset G$, the subgroup generated
    by $S$ is denoted by $\langle S \rangle=\langle s_1,\ldots,
    s_n \rangle$. If $x, \, y \in G$, their commutator is defined as
    $[x,\, y]=xyx^{-1}y^{-1}$.  The commutator subgroup of $G$ is denoted by $[G, \, G]$,
    the center
    of $G$ by  $Z(G)$. If $N$ is a normal subgroup of $G$ and $g \in G$, we denote by
  $\bar{g}$
    the image of $g$ in the quotient group $G/N$.  We denote by $\mZ_n$ the cyclic group of order $n$.
    $\mathrm{IdSmallGroup}(G)$ indicates the label of the group
    $G$ in
    the  \verb|GAP4| database of small groups. For instance,
    $\mathrm{IdSmallGroup}(\mathbb{Z}_6)=G(6, \, 2)$ means
    that $\mathbb{Z}_6$
   is in the second  position of the 
    list of groups of order $6$.

\section{Group-theoretical preliminaries: CCT-groups and monolithic groups} \label{sec:group-prel}

\subsection{CCT-groups}
This subsection relies  on \cite[Section 2]{PolSab22}, to which we refer the reader for more details. Further references are  \cite{Suz61, Schm70, Reb71, Rocke73, Wu98}.

\begin{definition} \label{def:CCT}
  A non-abelian finite group $G$ is said to be a
  \emph{center commutative-transitive group}
  $($or a CCT-group, for short$)$ if commutativity is a transitive
  relation on the set on non-central elements of $G$. In other words, if $x,
  \, y, \, z	\in G \setminus Z(G)$  and $[x, \, y]=[y, \, z]=1$, 
  then $[x, \, z]=1$. Equivalently, $G$ is a CCT-group if the centralizer $C_G(x)$ is abelian  for every non-central element $x \in G$. 
\end{definition}

\begin{proposition} \label{prop:small-CCT}
  Let $G$ be a non-abelian finite group.
  \begin{itemize}
    \item[$\boldsymbol{(1)}$] If $|G|$ is the product of at
      most three prime
      factors $($non necessarily distinct$)$, then $G$
      is a \emph{CCT}-group.
    \item[$\boldsymbol{(2)}$] If $|G|=p^4$, with $p$ prime,
      then $G$ is a
      \emph{CCT}-group.
    \item[$\boldsymbol{(3)}$] If $G$ contains an abelian normal
      subgroup of
      prime index, then $G$ is a \emph{CCT}-group.
  \end{itemize}
\end{proposition}

Non-CCT groups of order at most $32$ were classified in \cite{PolSab22}. They 
consist of the symmetric group $\mathsf{S}_4$ and the two extra-special groups of 
order $2^5 = 32$, namely $\mathsf{H}_5(\mathbb{Z}_2)=G(32, \,49)$ and $\mathsf{G}_5(\mathbb{Z}_2)=G(32, \, 50)$; see  Example \ref{ex:extra-special-is-monolithic} for the definition of extra-special group.
 This classification was later extended to groups of order at most $64$ in \cite{PolSab24}. In the present work, we need to consider the case of groups of order at most $127$. First of all, looking at Proposition \ref{prop:small-CCT}, we get 

\begin{proposition} \label{prop:CCT-order-less-127}
Let $G$ be a non-abelian group with $65 \leq |G| \leq 127$. If $G$ is not a \emph{CCT}-group, then 
\begin{equation}
|G|\in \{ 72, \, 80, \, 84, \, 88, \, 90, \, 96, \, 100, \, 104, \, 108, \, 112, \, 120, \, 126\}.
\end{equation}
\end{proposition}  

We now examine each of these cases, using the computer algebra system \verb|GAP4|, see \cite{GAP4}.

\begin{proposition} \label{prop:non-CCT}

The non-abelian groups $G$ of order $65 \leq |G| \leq 127$ which are not $\operatorname{CCT}$-groups are the following.
\begin{itemize}
\item The groups of the form $G(72, \, t)$, with
    \footnotesize
\begin{equation}
t \in \{15, \, 20, \, 21, \, 22, \,  23, \, 24, \, 40, \, 41, \, 42, \, 43, \, 44, \, 46\}.
\end{equation}
\normalsize
\item The groups of the form $G(80, \, t)$, where 
    \footnotesize
\begin{equation}
t \in \{15, \, 16, \, 17, \, 18, \, 29, \, 31, \, 33, \, 34, \, 39, \, 40, \, 41, \, 42\}.
\end{equation}
\normalsize
\item The group $G(84, \, 8)$.
\item The groups of the form $G(96, \, t)$, where 
    \footnotesize
\begin{equation}
\begin{split}
t \in \{& 3, \, 13, \, 14, \, 15, \, 16, \, 17, \, 29, \, 30, \, 31, \, 32, \, 33, \, 34, \, 35, \, 36, \, 39, \, 40, \, 41, \, 42, \, 43, \, 44, \, 49, \, 50,\\
& 51, \, 64, \, 65, \, 70, \, 71, \, 72, \, 84, \, 85, \, 86, \, 87, \, 88, \, 89, \, 90, \, 91, \, 92, \, 93, \, 94, \, 95, \, 96, \, 97, \, 98, \, 99, \\
& 101, \, 102, \, 103, \, 104, \, 105, \, 113, \, 114, \, 115, \, 116, \, 117, \, 118, \, 119, \, 120, \, 121, \, 122, \, 123, \, 124, \\
& 125, \, 126, \, 138, \, 139, \, 140, \, 141, \, 142, \, 143, \, 144, \, 145, \,  146, \, 147, \, 148, \, 149, \, 150, \, 151, \, 152, \\\
& 153, \,  154, \, 155, \, 156, \, 157, \, 158, \,  183, \, 184, \, 185, \, 186, \, 187, \, 190, \, 191, \, 
193, \, 194, \, 195, \, 197,\\
& 199, \, 201, \, 202, \, 203, \, 204, \, 209, \, 210, \, 211, \, 212, \, 213, \, 214, \, 215, \, 
216, \, 217, \, 224,  \, 225, \, 226, \, 227\}.
\end{split}
\end{equation}
\normalsize

\item The groups $G(100, \, 10)$ and $G(100, \, 13)$. 
\item The groups of the form $G(108, \, t)$, where 
    \footnotesize
\begin{equation}
t \in \{8, \, 9, \,  16, \, 17, \, 25, \, 26, \, 37, \,   38,\,  39, \,  40\}.
\end{equation}
\normalsize
\item The groups of the form  $G(112, \, t)$, where 
    \footnotesize
\begin{equation}
t \in \{14, \, 15, \, 16, \, 17, \, 31, \, 32, \, 33, \, 34\}.
\end{equation}
\normalsize
\item The groups of the form $G(120, \, t)$, where 
    \footnotesize
\begin{equation}
t \in \{7, \, 8, \, 9, \, 10, \, 11, \, 12, \, 13, \, 14, \, 34, \, 36, \, 37, \, 38, \, 39, \, 41, \, 42\}.
\end{equation}
\normalsize
\item  The groups $G(126, \, 8)$ and $G(126, \, 9)$. 
\end{itemize}
\end{proposition}
\begin{proof}
We include the \verb|GAP4| code used in the case $|G|=72$. The remaining cases can be treated in exactly the same way by changing the value of the variable \verb|order| in the first line.

\begin{lstlisting}[]
order:=72;;
#creation of list_nonab
list_nonab:=[];;
for t in [1..NumberSmallGroups(order)] do
G:=SmallGroup(order, t);
if IsAbelian(G)=false then
Add(list_nonab, IdSmallGroup(G));
fi; od; 

#creation of list_nonab_non_cct
list_nonab_non_cct:=[];;
for t in list_nonab do
G:=SmallGroup(t);
   #non-central elements in G
    H:=[];;
    for g in G do
    boole:=g in Center(G);
    if boole=false then
    Add(H, g);
    fi; od;
    #checking CCT property via centralizers
    ab:=0;; non_ab:=0;;
    for h in H do
    C:=Centralizer(G, h);
    boole:=IsAbelian(C);
    if boole=true then ab:=ab+1;
    else non_ab:=non_ab+1;
    fi; od;
if ab < Size(H) then
Add(list_nonab_non_cct, t); 
fi; od;
Size(list_nonab); list_nonab_non_cct;
\end{lstlisting}

The output is 
\begin{lstlisting}
44
[ [ 72, 15 ],[ 72, 20 ],[ 72, 21 ],[ 72, 22 ], [ 72, 23 ],[ 72, 24 ], 
  [ 72, 40 ],[ 72, 41 ],[ 72, 42 ],[ 72, 43 ],[ 72, 44 ], [ 72, 46 ] ]
\end{lstlisting}
hence, up to automorphisms, there are $44$ non-abelian groups of order $72$ and, among them,  those  which are not CCT-groups are precisely the ones  in the statement.
\end{proof}

\subsection{Monolithic groups}
Given a finite group $G$, we define  $\operatorname{mon}(G)$ as 
the intersection of all  non-trivial, normal subgroups of $G$.  The group $G$ is called \emph{monolithic} if  $\operatorname{mon}(G) \neq \{1\}$.  Equivalently, $G$ is monolithic if it contains precisely one minimal non-trivial, normal subgroup. 

%\begin{example} \label{ex:simple_monolithic}
%The group $G$ is simple if and only if $\operatorname{mon}(G)=G$. In particular, every non-trivial finite %simple group is monolithic.
%\end{example}

%\begin{example}\label{ex:product_non_monolithic}
%A product  $G=H \times K$, where $H$ and $K$ are non-trivial groups,  is never monolithic: in fact, $H %\times \{1\}$ and $\{1\} \times K$ are two non-trivial normal subgroups intersecting only at the identity of %$G$.   
%\end{example}

\begin{example} \label{ex:extra-special-is-monolithic}
A finite $p$-group $G$ is called  \emph{extra-special} if its center $Z(G)$ is cyclic of order $p$  and the quotient $V=G/Z(G)$ is a non-trivial, elementary abelian $p$-group.  Every extra-special group $G$  is monolithic, with $\operatorname{mon}(G)=Z(G)$. Indeed, since $Z(G) \simeq \mathbb{Z}_p$ is normal in $G$, we have $\operatorname{mon}(G) \subseteq Z(G)$. On the other hand, every non-trivial, normal subgroup of an  extra-special group contains the center (see \cite[Exercise 9
  p. 146]{Rob96}), hence $Z(G) \subseteq \operatorname{mon}(G)$.
\end{example}

Actually, a more general result holds true:

\begin{proposition} \label{prop:Mon-if-center-non-trivial}
If $G$ is monolithic and $Z(G) \neq \{1\}$, then $\operatorname{mon}(G)$ is cyclic of prime order.  In particular, if $G$ is a monolithic $p$-group then $\operatorname{mon}(G) \simeq \mathbb{Z}_p$.
\end{proposition}
\begin{proof}
If $Z(G)$ is non-trivial, by definition we have $\operatorname{mon}(G) \subseteq Z(G)$, hence  $\operatorname{mon}(G)$ is abelian and all its subgroups are normal in $G$. Thus the conclusion follows by using the minimality of  $\operatorname{mon}(G)$ and the fact that a finite $p$-group has a non-trivial center.
\end{proof}

In \cite{PolSab24} we obtained a classification of the non-abelian, monolithic groups of order up to $64$ that are not CCT-groups. We now broaden this analysis to include all groups of order at most $127$.

\begin{proposition} \label{prop:monolithic-non-CCT-order-less-127}
Let $G$ be a  non-abelian, monolithic group of order $65 \leq |G| \leq 127$. If $G$ is not a \emph{CCT}-group, then $G$ is isomorphic to one of the following$:$ 
\footnotesize
\begin{equation}
\begin{split}
& G(72, \, 40), \; G(72, \, 41), \\ 
& G(96,\, 64), \; G(96, \, 70), \; G(96,\, 71), \; G(96,\, 72), \; G(96, \, 190), \; G(96,\, 191), \; G(96, \, 193), \\ &  G(96, 201), \; G(96, \, 202), \; G(96, \, 204), \\ & 
G(108, 17), \\ &  G(120, \, 34).
\end{split}
\end{equation}
\normalsize
\end{proposition}
\begin{proof}
We again rely on \verb|GAP4|. To illustrate the procedure, we focus on the case 
$|G| = 72$; the analysis  for the remaining values of $|G|$ can be carried out in exactly the same way. We first run the script used in the proof of 
Proposition \ref{prop:non-CCT}, which produces the list 
\verb|list_nonab_non_cct| of all non-abelian, non-CCT groups of order $72$. We then 
execute the following script, which selects from this list the groups that are  monolithic:

\begin{lstlisting} 
#iteration checking monolithicy property 
list_mon:=[];;
for t in list_nonab_non_cct do
G:=SmallGroup(t);
N:=NormalSubgroups(G);;
Np:=[];;
for n in N do
if Order(n)>1 then
Add(Np, n);
fi; od;
M:=G;
for n in Np do
M:=Intersection(M, n);
od;
if Order(M)>1 then 
Add(list_mon, IdSmallGroup(G));
fi; od;
list_mon;
\end{lstlisting}
The output is
\begin{lstlisting}
[ [ 72, 40 ], [ 72, 41 ] ]
\end{lstlisting}
as this proves our claim.

   \end{proof}

\section{Pure surface braid groups and admissible braid quotients}
\label{sec:pure-braids}

For more details about the content of this section, see \cite{PolSab23}. Let $\Sigma_b$ be a compact Riemann surface of genus $b \geq 2$, and let
$\mathscr{P} = (p_1, \, p_2)$ be an ordered pair of distinct points on
$\Sigma_b$. A \emph{pure braid} (on two strands) on $\Sigma_b$ based at $\mathscr{P}$
is a pair $(\alpha_1, \, \alpha_2)$ of paths $\alpha_i \colon [0, \, 1]
\to \Sigma_b$ such that
\begin{itemize}
	\item $\alpha_i(0) = \alpha_i(1)=p_i \quad \textrm{for all }i \in
		\{1,\, 2\}$
	\item the points $\alpha_1(t), \, \alpha_2(t) \in \Sigma_b$
		are pairwise
		distinct for all $t \in [0, \, 1].$
\end{itemize}

\begin{definition} \label{def:braid}
	The \emph{pure braid group} on two strands on $\Sigma_b$ is the group
	$\mathsf{P}_{2}(\Sigma_b)$ whose elements are the pure braids
	based at
	$\mathscr{P}$ and whose operation is the usual concatenation of paths,
	up to homotopies among braids.
\end{definition}
It can be shown that $\mathsf{P}_{2}(\Sigma_b)$ does not depend on the choice
of the set $\mathscr{P}=(p_1, \, p_2)$, and that there is an isomorphism
\begin{equation} \label{eq:iso-braids}
	\mathsf{P}_{2}(\Sigma_b) \simeq \pi_1(\Sigma_b \times \Sigma_b
		- \Delta,
	\, \mathscr{P})
\end{equation}
where $\Delta \subset \Sigma_b \times \Sigma_b$ is the diagonal.

The group $\mathsf{P}_{2}(\Sigma_b)$  is finitely presented for all $b$,
and explicit presentations can be found in  \cite{Bir69, Scott70, Bel04, GG04}. 
Here we consider a system of  $4b+1$ generators described in \cite[Sections 1-3]{GG04}, 
referring the reader to that paper for further details. For all $j \in \{1, \ldots, b\}$, 
let us consider  the $4b$ elements
\begin{equation} \label{eq:braid-generators}
	\rho_{1j}, \; \tau_{1 j}, \; \rho_{2j}, \; \tau_{2 j}
\end{equation}
of $\mathsf{P}_{2}(\Sigma_b)$ represented by the pure braids shown in
Figure \ref{fig1}.

%\begin{figure}[H]
%	\begin{center}
%		\includegraphics*[totalheight=1.5 cm]{fig1}
%		\caption{The pure braids $\rho_{1j}, \, \tau_{1j}$,
%			$\rho_{2j}, \, \tau_{2j}$
%		on $\Sigma_b$} \label{fig1}
%	\end{center}
%\end{figure}

\begin{figure}[H]
    \centering
    \begin{subfigure}[b]{0.35\textwidth}
        \includegraphics*[totalheight=1.5 cm]{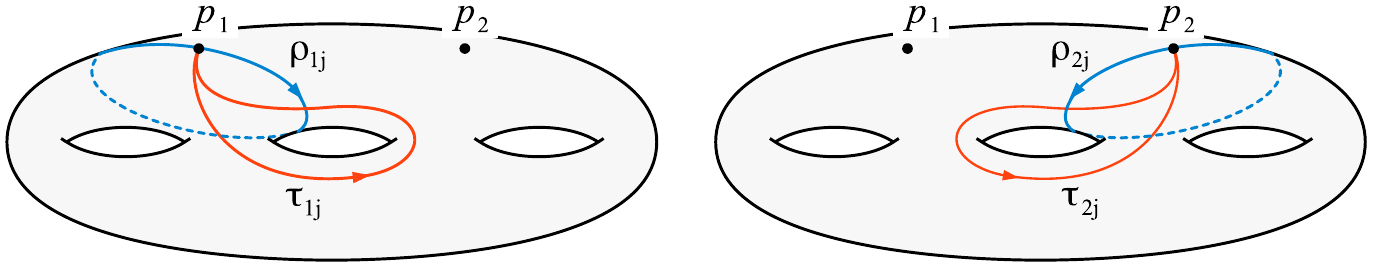}
        \caption{$\rho_{1j}, \, \tau_{1j}$,
			$\rho_{2j}, \, \tau_{2j}$}
        \label{fig1}
    \end{subfigure}
    \hfill
    \begin{subfigure}[b]{0.35\textwidth}
        	\includegraphics*[totalheight=1 cm]{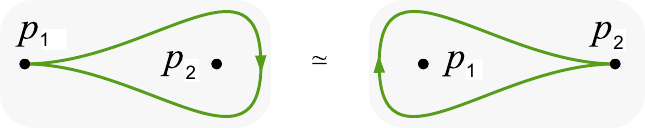}
            \caption{ $A_{12}$}
            \label{fig2}
    \end{subfigure}
    \caption{Generators of $\mathsf{P}_{2}(\Sigma_b)$}
\end{figure}

If $\ell \neq i$, the path corresponding to $\rho_{ij}$ and $\tau_{ij}$
based at $p_{\ell}$ is the constant path. Moreover, let $A_{12}$ be
the pure braid shown in Figure \ref{fig2}. In terms of the isomorphism
\eqref{eq:iso-braids}, the generators $\rho_{ij}$, $\tau_{ij}$ correspond to
the generators of $\pi_1(\Sigma_b \times \Sigma_b - \Delta, \, \mathscr{P})$
coming from the usual description of $\Sigma_b$ as the identification space
of a regular $4b$-gon, whereas $A_{12}$ corresponds to the homotopy class
in $\Sigma_b \times \Sigma_b - \Delta$ of a topological loop in $\Sigma_b
\times \Sigma_b$ that ``winds once'' around $\Delta$. Then the elements
\begin{equation*}
\rho_{11}, \, \tau_{11}, \ldots, \rho_{1b}, \, \tau_{1b}, \,\rho_{21}, \, \tau_{21}, \ldots, \rho_{2b}, \, \tau_{2b}, \, A_{12} 
\end{equation*}
generate $\mathsf{P}_2(\Sigma_b)$. We can now define the objects studied in this paper.

%\begin{figure}[H]
%	\begin{center}
%		\includegraphics*[totalheight=1 cm]{fig2}
%		\caption{The pure braid $A_{12}$ on $\Sigma_b$} \label{fig2}
%	\end{center}
%\end{figure}

\begin{definition} \label{def:pure-braid-quotient}
	Let $b, \, n \geq 2$ be integers, and let $G$ be a finite group. A group
	epimomorphism
	\begin{equation} \label{eq:group-epimorphism}
		\varphi \colon \mathsf{P}_2(\Sigma_b) \to G
	\end{equation}
	is called an \emph{admissible epimorphism of type} $(b, \, n)$ if $\varphi(A_{12})$ has order $n$.
	In this case, $G$ will be called an \emph{admissible $($pure$)$ braid quotient of type} $(b, \, n)$.
\end{definition}

\begin{remark}\label{rmk:purebraidquotient_nonabelian}
     The inclusion map $\iota \colon \Sigma_b \times \Sigma_b - \Delta
	\to \Sigma_b \times \Sigma_b$ induces a group epimorphism 
\begin{equation}	
	\iota_*
	\colon \mathsf{P}_2(\Sigma_b) \to \pi_1(\Sigma_b \times \Sigma_b,
	\, \mathscr{P}),
\end{equation}	
	 whose kernel is the normal closure of the subgroup
	generated by $A_{12}$. Thus, a group homomorphism $\varphi
	\colon
	\mathsf{P}_2(\Sigma_b) \to G$  is admissible if and only if $\varphi$ does not factor through 
	$\pi_1(\Sigma_b \times \Sigma_b, \, \mathscr{P})$.
	\end{remark}

\section{From Topology to Combinatorial Group Theory: Diagonal double Kodaira Structures}  \label{sec:DDKS}
\subsection{Diagonal double Kodaira structures} \label{subsec:ddks} 
Let us now translate the problem of constructing admissible epimorphisms $\varphi \colon \mathsf{P}_2(\Sigma_b) \to G$ into the problem of detecting some combinatorial structures on the finite group $G$. For more details about the content of this section, we refer the reader to \cite{Pol22, PolSab22, PolSab24}. 

Let $G$ be a finite group and let $b, \, n \geq 2$ be two positive integers.  A \emph{diagonal double Kodaira structure} of type $(b, \, n)$ on $G$
  is an ordered set of	$4b+1$ generators
  \begin{equation} \label{eq:ddks}
    \S = (\r_{11}, \, \t_{11}, \ldots, \r_{1b}, \, \t_{1b}, \;
      \r_{21}, \,
    \t_{21}, \ldots, \r_{2b}, \, \t_{2b}, \; \z ),
  \end{equation}
  with $o(\z)=n$, which are the images of the ordered set of generators
  \begin{equation}
  (\rho_{11}, \, \tau_{11}, \ldots, \rho_{1b}, \, \tau_{1b}, \; \rho_{21}, \, \tau_{21}, \ldots, \rho_{2b}, \, \tau_{2b}, \, A_{12})
  \end{equation}
  via an admissible braid epimorphism $\varphi \colon \mathsf{P}_2(\Sigma_b) \to G$ of type $(b, \, n)$.  
By definition, the classification of admissible braid quotients is equivalent to the classification of finite groups that admit diagonal double Kodaira structures. In the ensuing sections, we complete this classification for groups of order at most $127$.

\begin{remark} \label{rmk:why-ddks}
The term \emph{diagonal double Kodaira structure} originates from algebraic geometry. It was introduced by the first author in \cite{Pol22}, as the existence of such a structure of type $(b,\, n)$ on a finite group $G$ is equivalent to the existence of a $G$-cover
\begin{equation}
\mathbf{f} \colon S \to \Sigma_b \times \Sigma_b,
\end{equation}
branched with order $n$ along the diagonal $\Delta \subset \Sigma_b \times \Sigma_b$. This implies that $S$ is a \emph{double Kodaira surface}, namely, a complex surface of general type endowed with two distinct smooth and non-isotrivial fibrations, see \cite{CatRol09}. 
\end{remark}

The precise relations satisfied by the elements of $\mathfrak{S}$ can be found in \cite[Section 3]{PolSab22}. Here we just write down the special case of a diagonal double Kodaira structure 
of type $(2,\,n)$, that we will use later. It consists of an ordered set of nine generators of $G$,
\begin{equation}
\S=  (\r_{11}, \, \t_{11}, \,  \r_{12}, \, \t_{12}, \, \r_{21}, \,
    \t_{21}, \,
  \r_{22}, \, \t_{22}, \, \z ),
\end{equation}
with $o(\z)=n \geq 2$, subject to the following relations:
\scriptsize
\reqnomode
\begin{equation} \label{eq:ddks-genus-2}
  \begin{aligned}
    \mathbf{(S1)} & \,\, [\r_{12}^{-1}, \, \t_{12}^{-1}] \,
    \t_{12}^{-1} \,
    [\r_{11}^{-1}, \, \t_{11}^{-1}] \, \t_{11}^{-1}\, (\t_{11}
    \, \t_{12}) =
    \z & \\ \mathbf{(S2)} & \, \, [\r_{21}^{-1}, \, \t_{21}]
    \; \t_{21} \;
    [\r_{22}^{-1}, \, \t_{22}] \, \t_{22}\, (\t_{22}^{-1} \,
    \t_{21}^{-1})=
    \z^{-1} \\
    & \\
    \mathbf{(R1)} & \, \, [\r_{11}, \, \r_{22}]=1 & \mathbf{(R6)}
    & \, \,
    [\r_{12}, \, \r_{22}]=1 \\
    \mathbf{(R2)} & \, \, [\r_{11}, \, \r_{21}]=1 &
    \mathbf{(R7)} & \, \, [\r_{12}, \, \r_{21}]=
    \z^{-1}\,\r_{21}\,\r_{22}^{-1}\,\z\,\r_{22}\,\r_{21}^{-1} \\
    \mathbf{(R3)} & \, \, [\r_{11}, \, \t_{22}]=1 & \mathbf{(R8)}
    & \, \,
    [\r_{12}, \, \t_{22}]=\z^{-1} \\
    \mathbf{(R4)} & \, \, [\r_{11}, \, \t_{21}]=\z^{-1} &
    \mathbf{(R9)} & \, \,
    [\r_{12}, \, \t_{21}]=[\z^{-1}, \, \t_{21}] \\
    \mathbf{(R5)} & \, \, [\r_{11}, \, \z]=[\r_{21}^{-1}, \,
    \z] & \mathbf{(R10)}
    & \, \, [\r_{12}, \, \z]=[\r_{22}^{-1}, \, \z] \\
    & \\
    \mathbf{(T1)} & \, \, [\t_{11}, \, \r_{22}]=1 & \mathbf{(T6)}
    & \, \,
    [\t_{12}, \, \r_{22}]= \t_{22}^{-1}\, \z \, \t_{22} \\
    \mathbf{(T2)} & \, \, [\t_{11}, \, \r_{21}]=\t_{21}^{-1}\,
    \z \, \t_{21}
    & \mathbf{(T7)} & \, \, [\t_{12}, \, \r_{21}]= [\t_{22}^{-1},
    \, \z] \\
    \mathbf{(T3)} & \, \, [\t_{11}, \, \t_{22}]=1 & \mathbf{(T8)}
    & \, \,
    [\t_{12}, \, \t_{22}]=[\t_{22}^{-1}, \, \z] \\
    \mathbf{(T4)} & \, \, [\t_{11}, \, \t_{21}]=[\t_{21}^{-1},
    \, \z] &
    \mathbf{(T9)} & \, \, [\t_{12}, \, \t_{21}]=
    \t_{22}^{-1}\, \z \,\t_{22} \,\z^{-1} \,\t_{21} \,\z
    \,\t_{22}^{-1}\,\z^{-1}
    \,\t_{22}\,\t_{21}^{-1} \\
    \mathbf{(T5)} & \, \, [\t_{11}, \, \z]=[\t_{21}^{-1}, \,
    \z] & \mathbf{(T10)}
    & \, \, [\t_{12}, \, \z]=[\t_{22}^{-1}, \, \z] \\
  \end{aligned}
\end{equation}
\leqnomode
\normalsize

\subsection{Prestructures} \label{subsec:prestructures}
To determine whether a finite group $G$ admits diagonal double Kodaira structures, we first search for certain related combinatorial objects, which we call \emph{prestructures}.

\begin{definition} \label{ref:prestructure} Let $n \geq 2$. A \emph{prestructure of type} $n$ 
  on $G$ is an ordered set of nine elements
  \begin{equation} \label{eq:prestructure}
   \P= (\r_{11}, \, \t_{11}, \,  \r_{12}, \, \t_{12}, \, \r_{21},
      \, \t_{21}, \,
    \r_{22}, \, \t_{22}, \, \z ),
  \end{equation}
  with $o(\z)=n$, subject to the relations $(\mathrm{R1}), \ldots,
  (\mathrm{R10})$, $(\mathrm{T1}), \ldots, (\mathrm{T10})$ in
  \eqref{eq:ddks-genus-2}.
\end{definition}

In other words, the nine elements in $\mathfrak{P}$ must satisfy all the relations defining a diagonal double Kodaira structure of type $(2,\, n)$, except for the surface relations. Note that we do \emph{not} require the elements of a prestructure to generate $G$.

\begin{remark} \label{rmk:G-non-commutative}
If $G$ admits a prestructure of type $n$, then  $\mathsf{z}$ is a commutator of order $n \geq 2$ in $G$. Hence $G$ must be \emph{non-abelian} and $n$ divides the order of $[G, \, G]$.
%Moreover, $\r_{11}, \, \t_{11}, \,  \r_{12}, \, \t_{12}, \, \r_{21}, \, \t_{21}, \, \r_{22}, \, %\t_{22} $ are all non-central in $G$.
\end{remark}

\begin{remark} \label{remark:prestructure-structure}
If $G$ admits no prestructure of type $n$, then it admits no diagonal double Kodaira structure of type $(b,n)$ for any $b$. In particular, if $G$ admits no prestructure at all, then it admits no diagonal double Kodaira structure of any type and therefore cannot occur as an admissible braid quotient.
\end{remark}

A key point is  now provided by the following

\begin{remark}\label{rmk:cct-no-prestructure}
A CCT-group admits no prestructure, see \cite[Proposition 4.4]{PolSab22}. Consequently, by Remark \ref{remark:prestructure-structure}, the search for diagonal double Kodaira structures can be restricted to non-CCT groups.

\end{remark}

We now explain the role played by monolithic groups in this setting. Let
\begin{equation}
\P= (\r_{11}, \, \t_{11}, \,  \r_{12}, \, \t_{12}, \, \r_{21},
      \, \t_{21}, \,
    \r_{22}, \, \t_{22}, \, \z ),
\end{equation}
be a prestructure of type $n$ on a finite group $G$. If $ \z \notin \operatorname{mon}(G)$, we can find a non-trivial normal subgroup $N$ of $G$ such that $\z \notin N$. Thus, the ordered set
 \begin{equation}
\bar{\P} =(\bar{\r}_{11}, \, \bar{\t}_{11}, \,  \bar{\r}_{12}, \, \bar{\t}_{12}, \, \bar{\r}_{21},
      \, \bar{\t}_{21}, \,
    \bar{\r}_{22}, \, \bar{\t}_{22}, \, \bar{\z} )
\end{equation}
is a prestructure of type $\bar{n}$ on the quotient group $G/N$, with $\bar{n}$ dividing $n$, and $\P$ is obtained as a lifting of  $\bar{\P}$ via the quotient homomorphism $\pi \colon G \to G/N$. Exactly the same argument applies to a diagonal double Kodaira structure $\S$  of type $(b, \,n)$. This yields the result below, cf.~\cite[Proposition 4.7]{PolSab22}.
\begin{proposition} \label{prop:monolithic-argument} If $G$ is a finite group, then the following holds.
\begin{itemize}
\item[$\boldsymbol{(1)}$ ] Assume that  $G$ admits a prestructure  $\mathfrak{P}$ of type $n$.  Then either $G$ is monolithic, and $\z \in \operatorname{mon}(G)$, or  there exist a proper quotient $H$ of $G$ and a prestructure $\bar{\P}$ of type $\bar{n}$ on $H$, with $\bar{n}$ dividing $n$, such that $\P$ is obtained as a lifting of $\bar{\P}$ via the quotient homomorphism $\pi \colon G \to H$.
\item[$\boldsymbol{(2)}$ ] Assume that  $G$ admits a double Kodaira structure  $\mathfrak{S}$ of type $(b, \, n)$.  Then either $G$ is monolithic, and $\z \in \operatorname{mon}(G)$, or  there exist a proper quotient $H$ of $G$ and a prestructure $\bar{\S}$ of type $(b, \, \bar{n})$ on $H$, with $\bar{n}$ dividing $n$, such that $\S$ is obtained as a lifting of $\bar{\S}$ via the quotient homomorphism $\pi \colon G \to H$.
\end{itemize}
\end{proposition}

Accordingly, diagonal double Kodaira structures occurring in monolithic groups can be viewed as the ``genuinely new'' ones.

\begin{corollary} \label{cor:prestructure}
Assume that $G$ admits a prestructure $[$resp. a diagonal double Kodaira structure$]$, while none of its proper quotients does. Then $G$ is monolithic and $\z$ belongs to $\operatorname{mon}(G)$. 
\end{corollary}

Even if a prestructure $\P$ of type $n$ in $G$ does not yield a diagonal double 
Kodaira structure of type $(2, \, n)$ on $G$—either because it does not generate $G$ 
or because the surface relations are not satisfied—it could a priori admit an 
extension to a diagonal double Kodaira structure of type $(b, \, n)$ with $b \ge 3$. 
A necessary condition for this to occur is given by the following 

\begin{proposition} \label{prop:prestructure-b-3}
Let 
\begin{equation}
\P= (\r_{11}, \, \t_{11}, \,  \r_{12}, \, \t_{12}, \, \r_{21},
      \, \t_{21}, \,
    \r_{22}, \, \t_{22}, \, \z ),
\end{equation}
be a prestructure of type $n$ on a finite group $G$, and denote by $K$ the subgroup of $G$ generated  
by the elements $\r_{11}, \, \t_{11}, \, \r_{12}, \, \t_{12}$. If $\P$ can be extended to a diagonal double Kodaira structure $\S$ of type $(b, \, n)$, with $b \geq 3$, then the centralizer  $C_G(K)$ strictly contains the center $Z(G)$.
\end{proposition} 
\begin{proof}
Looking at the sets of relations for a diagonal double Kodaira structure given in \cite[Section 3]{PolSab22}, we see that $\r_{23} \in \S$ satisfies 
\begin{equation}
[\t_{13}, \, \r_{23}] = \t_{23}^{-1} \,\z \,\t_{23}, \quad [\r_{11}, \, \r_{23}]=[\r_{12}, \r_{23}]=[\t_{11}, \, \r_{23}]=[\t_{12}, \r_{23}]=1, 
\end{equation}
hence it provides a non-central element in $C_G(K)$.
\end{proof}

%As a consequence, we get the following result, that can be used to rule out the existence of diagonal % double Kodaira structures of type $(b, \, n)$ with $b \geq 3$.

\begin{corollary} \label{cor:prestructure-b-3}
Let $G$ be a finite group such that, for every prestructure $\P$ in $G$, the subgroup $K=\langle \r_{11}, \, \t_{11}, \, \r_{12}, \, \t_{12} \rangle$ satisfies $C_G(K)=Z(G)$. Then every diagonal double Kodaira structure on $G$ $($if any$)$ is necessarily of type $(2, \, n)$. Equivalently, any admissible epimorphism onto $G$ $($if it exists$)$ must be of the form $\varphi \colon \mathsf{P}_2(\Sigma_2) \to G$.
\end{corollary}

\section{Diagonal double Kodaira structures in groups of order at most $127$} \label{sec:most-127}

\subsection{Diagonal double Kodaira structures in groups of order at most $64$}
\label{subsec:order-up-to-64}  

Using the ideas outlined above, in \cite{PolSab22, PolSab24} we established the following result:

\begin{proposition} \label{prop:32-64}
If $G$ admits prestructures and $|G| \le 64$, then either $|G| = 32$ or $|G| = 64$. Consequently, if $|G| \leq 64$ and $|G| \neq 32,64$, then $G$ admits no diagonal double Kodaira structure of any type, and therefore cannot occur as an admissible braid quotient.
\end{proposition}

In those papers, we also provided a complete classification of the groups of orders $32$ and $64$ that admit diagonal double Kodaira structures of type $(b, \, n)=(2, \, 2)$, together with the number of such structures in each case. In the present work, we extend this analysis by describing all groups $G$ with $|G| \le 64$ that admit prestructures, and by proving that none of them admits diagonal double Kodaira structures of type $(b, \, n)$ with $b \geq 3$.

\begin{proposition} \label{prop:prestructures up to order 64}
Let $G$ be a finite group with $|G| \leq 64$ that admits a prestructure of type $n$. 
Then necessarily $n = 2$, and $G$ is one of the groups listed in Table $\ref{tab:32_64}$ below. Moreover, none 
of these groups admits diagonal double Kodaira structures of type $(b, \, n)$ with 
$b \ge 3$.

\begin{table}[H]  
    \scriptsize
%\caption{Groups of order $\leq 64$ admitting prestructures}  
    \centering
    \begin{tabularx}{0.75\linewidth}{@{}cccc@{}} 
\toprule
& & & $\mathrm{Number} \, \, \mathrm{of} \,\, \mathrm{diagonal} $\\
& & & $ \mathrm{double} \, \, \mathrm{Kodaira}\, \, \mathrm{structures}$     \\
 & & $\mathrm{Number} \, \, \mathrm{of} \,\,  \mathrm{prestructures}  $  &  $\mathrm{of} \,\, \mathrm{type} \,\, (b, \, n)= (2, \, 2)  $  \\
      $\mathrm{IdSmallGroup}(G)$ & $\mathrm{Is} \, \, G\, \,  \, \mathrm{monolithic}?$ & $ \mathrm{in}\, \, G\,\,\mathrm{up} \, \, \mathrm{ to} \, \, \operatorname{Aut}(G)$ & $ \mathrm{in}\, \, G\,\,\mathrm{up} \, \, \mathrm{ to} \, \, \operatorname{Aut}(G)$ \\
      \toprule
$G(32, \, 49)$ & $\mathrm{Yes}$ & $4480$ & $1920$ \\    
$G(32, \, 50)$ & $\mathrm{Yes}$ & $2688$ & $1152$ \\       
%\hline
$G(64, \, 134)$ & $\mathrm{Yes}$ & $40320$ & $0$ \\        
$G(64, \, 135)$ & $\mathrm{Yes}$ & $40320$ & $0$ \\     
$G(64, \, 136)$ & $\mathrm{Yes}$ & $40320$ & $0$ \\  
$G(64, \, 137)$ & $\mathrm{Yes}$ & $40320$ & $0$ \\  
$G(64, \, 138)$ & $\mathrm{Yes}$ & $26880$ & $0$ \\  
$G(64, \, 139)$ & $\mathrm{Yes}$ & $26880$ & $0$ \\  
$G(64, \, 199)$ & $\mathrm{No}$ & $322560$ & $138240$ \\  
$G(64, \, 200)$ & $\mathrm{No}$ & $107520$ & $46080$ \\ 
$G(64, \, 201)$ & $\mathrm{No}$ & $430080$ & $184320$ \\ 
$G(64, \, 249)$ & $\mathrm{Yes}$ & $860160$ & $368640$ \\ 
$G(64, \, 257)$ & $\mathrm{Yes}$ & $26880$ & $0$ \\ 
$G(64, \, 258)$ & $\mathrm{Yes}$ & $53760$ & $0$ \\ 
$G(64, \, 259)$ & $\mathrm{Yes}$ & $26880$ & $0$ \\ 
$G(64, \, 264)$ & $\mathrm{No}$ & $38080$ & $14400$ \\ 
$G(64, \, 265)$ & $\mathrm{No}$ & $22848$ & $8640$ \\ 
$G(64, \, 266)$ & $\mathrm{Yes}$ & $60928$ & $23040$ \\ 
\bottomrule
\end{tabularx} 
\caption{Groups of order at most $64$ with prestructures}
\label{tab:32_64} 
 \end{table} 

\end{proposition}
\begin{proof}
By Proposition \ref{prop:32-64}, it suffices to consider the cases $|G| = 32$ and $|G| = 64$.

The groups of order $32$ that admit prestructures were identified in \cite[Section 4]{PolSab22}: they are precisely the two extra-special groups, namely $\mathsf{H}_5(\mathbb{Z}_2)=G(32, \, 49)$ and $\mathsf{G}_5(\mathbb{Z}_2)=G(32,\, 50)$. Moreover, the number of prestructures and diagonal double Kodaira structures in each case can be computed either by using \verb|GAP4|, or by  the same computer-free techniques employed in the proof of \cite[Theorem 4.15]{PolSab22}.

Next, assume  that $G$ is a non-CCT group of order $64$. By Proposition \ref{prop:monolithic-argument}, if $G$ admits a prestructure either $G$ is monolithic (and $\z \in \operatorname{mon}(G) \simeq \mathbb{Z}_2$, see Proposition \ref{prop:Mon-if-center-non-trivial}), or $G$ admits a quotient $H$ which is extra-special of order $32$. The occurrences in both cases are listed in \cite{PolSab24}:
\begin{itemize}
\item If $G$ is monolithic, then $G$ is of the form $G(64, \, t),$ with $t$ in the set
\begin{equation} \label{eq:case_64_monolithic}
\begin{split}
\{ & 18, \, 19,\, 25,\, 28,\, 30,\, 32, \,33, \,34,\, 35, \,36,\, 37,\, 41,\, 42,\, 43,\, 46,
\, 91,\, 94, \,102,\, 111, \,125,\, 134,\, \\
   & 135,\, 136,\, 137,\, 138,\, 139, \,152, \,153,\, 154,\, 190,\, 191, \, 249, \,256, \, 257, \, 258,\, 259,\, 266\}.
\end{split}
\end{equation}
\item If  $G$ is non-monolithic and admits an extra-special quotient of order $32$, then $G$ is of the form $G(64, \, t),$ with $t$ in the set
\begin{equation} \label{eq:case_64_non_monolithic}
\begin{split}
\{ & 199, \,200,\, 201,\, 215, \, 216,\, 217,\, 218,\, 219, \,220,\, 221,\, 222,\, 223,\, 224,\, 225,\, 226,\, 227,\, 228, \,229, \\
& 230,\, 231, \, 232, \, 233, \, 234, \, 235, \, 236, \, 237, \, 238,\, 239, \, 240, \, 264, \, 265\}.
\end{split}
\end{equation}
\end{itemize}

In order to analyze all these cases, we implemented a \verb|GAP4| function called \verb|CheckStructures|, see Part 1 of our script in the GitHub repository \cite{GAP4DDKS}. Such a function provides a complete list of prestructures and diagonal double Kodaira structures of type $(2, \, n)$ for a given finite group
$G$ in the \verb|GAP4| database, modulo the action of $\operatorname{Aut}(G)$. 

%Finally, for each group $G$ in the table and for any prestructure
%\begin{equation}
%\P= (\r_{11}, \, \t_{11}, \,  \r_{12}, \, \t_{12}, \, \r_{21}, \, \t_{21}, \, \r_{22}, \, \t_{22}, \, \z )
%\end{equation}
% in $G$, one verifies (again using \verb|GAP4|) that the subgroup $K=\langle \r_{11}, \, \t_{11}, \,  %%\r_{12}, \, \t_{12} \rangle$ always satisfies $C_G(K)=Z(G)$. Consequently, Corollary \ref{cor:prestructure-%b-3} implies that every diagonal double Kodaira structure in $G$ is among the  structures of type $(b, \, %n)=(2,\, 2)$ listed in Table \ref{tab:32_64}.
 
For the reader’s convenience, we illustrate how \verb|CheckStructures| operates in the case $G=G(32,49)$; the other cases can be treated analogously.
\begin{lstlisting}
Read("Function_CheckStructures.g");;
G := SmallGroup(32,49);;
L := CheckStructures(G);;
\end{lstlisting} 
The output is
\begin{lstlisting}
Group ID: [ 32, 49 ] - Structure: (C2 x C2 x C2) : (C2 x C2)
Number of non-central elements of G = 30
Generating the list of five candidates ...
List of five candidates ready, size = 322560
...computing up to automorphism:
Elements modulo auto: 352
------------
Checking for prestructures now, 352 elements to go
Prestructures found 5632
...computing modulo automorphism
Elements modulo auto: 4480
Checking for structures now
Structures found: 1920
...computing up to automorphism:
Elements modulo auto: 1920
\end{lstlisting} 
and this shows that $G$ contains $1920$ diagonal double Kodaira structure of type $(b, \, n)$, up to automorphisms. Verifying that $n=2$ in every case is straightforward, either using \verb|GAP4| or noticing that $[G, \, G] \simeq \mathbb{Z}_2$. Finally, these are the only diagonal double Kodaira structures in $G$. To see this, it is sufficient to check that, for every prestructure in $G$, the subgroup $K=\langle \r_{11},\,\t_{11}, \, \r_{12}, \, \t_{12} \rangle$ satisfies $C_G(K) = Z(G)$, see Corollary \ref{cor:prestructure-b-3}:
\begin{lstlisting}
#Define the group G(32, 49)
H:=SmallGroup(32, 49);;
StructureDescription(H);
G:=Image(IsomorphismFpGroup(H));;
AssignGeneratorVariables(G);; 
RelatorsOfFpGroup(G);
f1:=F1;; f2:=F2;; f3:=F3;; f4:=F4;; f5:=F5;;
#Compute the order of the subgroup K for every list
#and check whether is centralizer equals the center of G
Read("List_prestructures_mod_32_49.g");
list:=List_prestructures_mod_32_49;;
list_K:=[];;
for l in list do
K:=Subgroup(G, [l[1], l[2], l[3], l[4]]);
C:=Centralizer(G, K);
AddSet(list_K, [Order(K), C=Center(G)]);
od;
list_K;
\end{lstlisting}   
The output is
\begin{lstlisting}
[ [32, true] ] 
\end{lstlisting}   
This means that for every prestructure in $G$ we have $K=G$, hence $C_G(K) = Z(G)$, as claimed. 
 
 \end{proof}

\subsection{Diagonal double Kodaira structures in groups of order between $65$ and $127$}
\label{subsec:order-up-to-127}  

We want now to extend Proposition \ref{prop:prestructures up to order 64} to groups $G$ of order $|G| \leq 127$. Let us start by ruling out the case where $|G| \neq 96$.

\begin{proposition} \label{prop:different-96}
If $65 \leq |G| \leq 127$ and $|G|\neq 96$, then $G$ admits no prestructure and, consequently, no diagonal double Kodaira structure  of any type.
\end{proposition}
\begin{proof}
If If $65 \leq |G| \leq 127$ and $|G|\neq 96$, then $G$ cannot have quotients of order $32$ or $64$ and so, by Corollary \ref{cor:prestructure} and Proposition \ref{prop:32-64}, if $G$ admits a prestructure then $G$ is monolithic and $\z \in \operatorname{mon}(G)$. Monolithic, non-CCT groups whose order is as in the statement are listed in Proposition \ref{prop:monolithic-non-CCT-order-less-127}. We can now perform a case-by-case analysis of these groups by using again our  \verb|GAP4| function \verb|CheckStructures|, and verify that the set of prestructures is empty for all of them.
\end{proof}

If $|G| = 96$, the conclusion of Proposition \ref{prop:different-96} no longer applies; however, the following alternatives hold.

\begin{proposition} \label{prop:equal-96}
Let $G$ be a finite group with $|G|=96$ and admitting a prestructure $\P$. 
\begin{itemize}
\item[$\boldsymbol{(1)}$] If $G$ is non-monolithic, then there exist an extra-special quotient $H$ of order $32$ and a prestructure $\bar{\S}$ on $H$ such that $\P$ is obtained as a lifting of $\bar{\P}$ via the quotient homomorphism $\pi \colon G \to H$.
\item[$\boldsymbol{(2)}$] If $G$ does not have extra-special quotients of order $32$, then $G$ is monolithic and $\z \in \operatorname{mon}(G)$.
\end{itemize}
Furthermore, analogous statements hold when $G$ admits a double Kodaira structure $\S$.
\end{proposition}
\begin{proof}
Part $\boldsymbol{(1)}$ follows directly from Propositions \ref{prop:monolithic-argument} and \ref{prop:32-64}, whereas part $\boldsymbol{(2)}$ is a consequence of Corollary \ref{cor:prestructure}. The proofs of the analogous statements when $G$ admits a diagonal double Kodaira structure $\S$ are exactly the same.
\end{proof}

In the rest of this section, we classify all groups of order $96$ that admit prestructures. Let us first consider the monolithic case.

\begin{proposition} \label{prop:96-monolithic}
Let $G$ be a monolithic group of order $96$ that admits a prestructure of type $n$. Then necessarily $n=2$ and $G$ is one of the groups in Table  $\ref{tab:96-monolithic}$ below.
Furthermore, none of these groups admits diagonal double Kodaira structures of any type. 
\begin{table}[H]
    \scriptsize
      \centering
    \begin{tabularx}{0.35\linewidth}{@{}cc@{}}
      \toprule
 & $\mathrm{Number} \, \, \mathrm{of} \,\,  \mathrm{prestructures}  $    \\
      $\mathrm{IdSmallGroup}(G)$ & $ \mathrm{in}\, \, G\,\,\mathrm{up} \, \, \mathrm{ to} \, \, \operatorname{Aut}(G)$  \\
      \toprule
    $G(96, \, 201)$ &  $8960$\\
        $G(96, \, 202)$ & $26880$ \\
$G(96, \, 204)$ &  $8960$\\
\bottomrule
\end{tabularx}
\caption{Monolithic groups of order $96$ with prestructures}  
\label{tab:96-monolithic}
\end{table}
\end{proposition}
\begin{proof}
Proposition \ref{prop:monolithic-non-CCT-order-less-127} shows that the non-abelian, non-CCT, monolithic groups of order $96$ are precisely the groups $G(96,\, t)$, with
\begin{equation}
t \in \{64,70,71,72,190,191,193,201,202,204\}.
\end{equation}
Using our \verb|GAP4| function \verb|CheckStructures|, we verify that prestructures occur only in the last three cases, and that their number, modulo the action of $\operatorname{Aut}(G)$, agrees with Table \ref{tab:96-monolithic}. The same computation shows that for all prestructures we have $n=2$, and that none of these groups admits diagonal double Kodaira structures of type $(b,\,n)=(2,\,2)$. Furthermore, for every prestructure the subgroup $K=\langle \r_{11},\,\t_{11}, \, \r_{12}, \, \t_{12} \rangle$ satisfies $C_G(K)=Z(G)$. By Corollary \ref{cor:prestructure-b-3}, it follows that none of the groups listed in Table \ref{tab:96-monolithic} admits diagonal double Kodaira structures of any type.
\end{proof}

It remains to consider non-monolithic groups of order $96$ admitting prestructures. By Proposition~\ref{prop:equal-96}, any prestructure in $G$ is lifted from an extraspecial quotient $H$ of order $32$.

\begin{proposition} \label{prop:96-non-monolithic}
Let $G$ be a non-monolithic group of order $96$ that admit a prestructure of type $n$. Then necessarily $n=2$ and the extra-special quotient $H$ is unique. All the occurring pairs $(G, \, H)$, together with the corresponding number of prestructures and diagonal double Kodaira structures of type $(b, \, n)=(2, \, 2)$, are listed in Table  $\ref{tab:96-non-monolithic}$ below. Furthermore, none of these groups admits diagonal double Kodaira structures of type $(b, \, n)$ with $b \geq 3$.
\begin{table}[H]
    \scriptsize
    \centering
 \begin{tabularx}{0.75\linewidth}{@{}cccc@{}} 
      \toprule
 & & & $\mathrm{Number} \, \, \mathrm{of} \,\, \mathrm{diagonal} $\\
& & & $ \mathrm{double} \, \, \mathrm{Kodaira}\, \, \mathrm{structures}$     \\
 & & $\mathrm{Number} \, \, \mathrm{of} \,\,  \mathrm{prestructures}  $  &  $\mathrm{of} \,\, \mathrm{type} \,\, (b, \, n)= (2, \, 2)  $  \\
      $\mathrm{IdSmallGroup}(G)$ & $\mathrm{IdSmallGroup}(H)$ & $ \mathrm{in}\, \, G\,\,\mathrm{up} \, \, \mathrm{ to} \, \, \operatorname{Aut}(G)$ & $ \mathrm{in}\, \, G\,\,\mathrm{up} \, \, \mathrm{ to} \, \, \operatorname{Aut}(G)$ \\
     \toprule
$G(96, \, 211)$ &  $G(32, \, 49)$ & $40320$ & $0$\\
 $G(96, \, 216)$ &  $G(32, \, 49)$ & $26880$ & $0$\\
 $G(96, \, 224)$ &  $G(32, \, 49)$ & $14698880$ & $6297600$\\
$G(96, \, 214)$ &  $G(32, \, 50)$ & $13440$ & $0$\\
$G(96, \, 217)$ &  $G(32, \, 50)$ &  $26880$ & $0$\\
$G(96, \, 225)$ &  $G(32, \, 50)$ & $8819328$ & $3778560$\\
\bottomrule
\end{tabularx}
\caption{Non-monolithic groups of order $96$ with prestructures}  
\label{tab:96-non-monolithic} 
 \end{table}
\end{proposition}
\begin{proof}
We begin by selecting the groups $G$ of order $96$ that satisfy the required properties and admit an extra-special quotient $H$ of order $32$. This can be achieved by using Proposition \ref{prop:non-CCT}, together with the following \verb|GAP4| script:

\begin{lstlisting}
#looking for non-CCT, non-monolithic groups of order 96 
#having extra-special quotients of order 32
T1:=[3, 13, 14, 15, 16, 17, 29, 30, 31, 32, 33, 34, 35, 36, 39, 40, 41, 42, 43, 44, 49, 50, 51, 64, 65, 70, 71, 72, 84, 85, 86, 87, 88, 89, 90, 91, 92, 93, 94, 95, 96, 97, 98, 99, 101, 102, 103, 104, 105, 113, 114, 115, 116, 117, 118, 119, 120, 121, 122, 123, 124, 125, 126, 138, 139, 140, 141, 142, 143, 144, 145, 146, 147, 148, 149, 150, 151, 152, 153, 154, 155, 156, 157, 158, 183, 184, 185, 186, 187, 190, 191, 193, 194, 195, 197,199, 201, 202, 203, 204, 209, 210, 211, 212, 213, 214, 215, 216,217, 224, 225, 226, 227];;
T2:=[64, 70, 71, 72, 190, 191, 193, 201, 202, 204];;
T:=Difference(T1, T2);
list96_49:=[];; list96_50:=[];;
for t in T do
G:=SmallGroup(96, t);
for N in NormalSubgroups(G) do
H:=G/N;
if IdSmallGroup(H)=[32, 49] then
AddSet(list96_49, IdSmallGroup(G)); 
fi;
if IdSmallGroup(H)=[32, 50] then
AddSet(list96_50, IdSmallGroup(G)); 
fi;
od; od;
list96_49; list96_50;
\end{lstlisting}
The output is
\begin{lstlisting}
[ [ 96, 211 ], [ 96, 216 ], [ 96, 224 ] ]
[ [ 96, 214 ], [ 96, 217 ], [ 96, 225 ] ]
\end{lstlisting}
We thus obtain precisely the groups listed in Table \ref{tab:96-non-monolithic}. A direct check with \verb|GAP4| shows that each of these groups admits a unique extraspecial quotient $H$. The number of prestructures and of diagonal double Kodaira structures of type $(2,\,n)$ is computed using our \verb|GAP4| function \verb|CheckStructures|, which also verifies that $n=2$ in every case. Finally, applying Corollary \ref{cor:prestructure-b-3} as above, we conclude that none of the groups listed in the table admits diagonal double Kodaira structures of type $(b,\, n)$ with $b\ge 3$.

\end{proof}

Recalling the one-to-one correspondence between diagonal double Kodaira structures and admissible braid quotients, we may now summarize the outcome of our investigation in the following 
\begin{theorem} \label{thm:main}
Let $G$ be an admissible braid quotient of type $(b, \, n)$. If $|G| \leq 127$, then necessarily  $(b, \, n)=(2, \, 2)$, and all the occurrences are listed in Table $\ref{tab:admissible_127}$ below. 
\begin{table}[H]  
    \scriptsize
%\caption{Groups of order $\leq 64$ admitting prestructures}  
    \centering
    \begin{tabularx}{0.75\linewidth}{@{}cccc@{}} 
\toprule
& &  $\mathrm{Number} \, \, \mathrm{of}$ &   $\mathrm{Total} \, \, \mathrm{number} \, \, \mathrm{of}$ \\
&&  $\mathrm{admissible} \, \, \mathrm{epimorphisms}$ &   $\mathrm{admissible} \, \, \mathrm{epimorphisms}$  \\

 &   &  $\varphi \colon \mathsf{P}_2(\Sigma_2) \to G $  &  $\varphi \colon \mathsf{P}_2(\Sigma_2) \to G $ \\
      $\mathrm{IdSmallGroup}(G)$ & $\mathrm{Is} \, \, G\, \,  \, \mathrm{monolithic}?$ &  
      $ \mathrm{up} \, \, \mathrm{ to} \, \, \operatorname{Aut}(G)$ & \\
      \toprule
$G(32, \, 49)$ & $\mathrm{Yes}$  & $1920$  & $2211840$\\    
$G(32, \, 50)$ & $\mathrm{Yes}$  & $1152$ & $2211840$ \\       
$G(64, \, 199)$ & $\mathrm{No}$  & $138240$ & $566231040$ \\  
$G(64, \, 200)$ & $\mathrm{No}$  & $46080$ & $566231040$\\ 
$G(64, \, 201)$ & $\mathrm{No}$ & $184320$ & $566231040$ \\ 
$G(64, \, 249)$ & $\mathrm{Yes}$  & $368640$ & $566231040$ \\ 
$G(64, \, 264)$ & $\mathrm{No}$  & $14400$ &  $530841600$\\ 
$G(64, \, 265)$ & $\mathrm{No}$  & $8640$ & $530841600$\\ 
$G(64, \, 266)$ & $\mathrm{Yes}$  & $23040$ & $530841600$\\ 
$G(96, \, 224)$ & $\mathrm{No}$ & $6297600$ & $14509670400$\\
$G(96, \, 225)$ & $\mathrm{No}$ & $3778560$ &  $14509670400$ \\
\bottomrule
\end{tabularx} 
\caption{Admissible braid quotients of order at most $127$}
\label{tab:admissible_127} 
 \end{table} 
\end{theorem}

\begin{corollary} \label{cor:b=3_128}
If $b \ge 3$, then every admissible quotient of $\mathsf{P}_2(\Sigma_b)$ has order at least $128$.
\end{corollary}

This bound is sharp: in fact,  the two extra-special groups of order $128$ occur as admissible quotients of  $\mathsf{P}_2(\Sigma_3)$, see Section \ref{sec:problems}.

\section{Final Considerations and open problems} \label{sec:problems}

Theorem \ref{thm:main} suggests that admissible braid quotients are extremely sparse. 
Indeed, among the $1036$ non-abelian groups of order at most $127$, only $11$ occur 
as admissible braid quotients. Let $a_m$ denote the number of admissible braid 
quotients $G$ with $|G| \le m$.

\begin{question}\label{q:asymptotic}
What can be said about the asymptotic behaviour of the sequence $(a_m)$?
\end{question}

For each $b \ge 2$, let $G_b$ denote the smallest possible order of an admissible 
quotient $G$ of $\mathsf{P}_2(\Sigma_b)$.

\begin{question}\label{q:G_b}
Can the integer $G_b$ be computed explicitly, or at least effectively bounded?
\end{question}

As a first example, by \cite{PolSab22} we know that $G_2 = 32$. Moreover, the results of \cite{CaPol19}, 
which show that certain extra-special groups arise as admissible braid quotients, 
yield the upper bound
\begin{equation}\label{eq:G_b}
G_b \le \min\bigl\{ p^{2b+1},\, 5^{4b+1} \bigr\},
\end{equation}
where $p$ denotes the smallest prime dividing $b+1$. However, it is 
unclear how sharp this bound is. Note that \eqref{eq:G_b} implies $G_b \le 2^{2b+1}$ whenever $b$ is odd, and in particular $G_3 \le 128$. On the other hand, Corollary \ref{cor:b=3_128} shows that 
$G_3 \ge 128$. It follows that $G_3 = 128$; in other words, for $b=3$ equality 
holds in \eqref{eq:G_b}.

%\begin{question}\label{q:unique_128}
%Are the two extra-special groups of order $128$ the only admissible quotients of
%$\mathsf{P}_2(\Sigma_3)$ having minimal order?
%\end{question}

%Further examples in higher genus would be needed to clarify whether extra-special 
%groups play a distinguished role for all values of $b$. We conclude the discussion of the invariant $G_b$ by %stating  the following 

\begin{question}\label{conj:G_b}
Is it true that, for every integer $b \ge 2$, we have $G_b = 2^{2b+1}$?
\end{question}

The evidence supporting  an affermative answer to Question \ref{conj:G_b} is currently limited to the cases $b=2$ and $b=3$. Furthermore, the computational techniques used in  this paper are not suited to a brute-force investigation for $b \ge 4$, 
suggesting that new conceptual tools will be needed to address the problem in 
higher genus.

\smallskip\smallskip
A further observation is that all admissible braid 
quotients discovered so far are nilpotent groups of class $2$, that is, they are 
non-abelian groups whose commutator subgroup is contained in the center.

% It would 
%be interesting to determine whether genuinely new families arise as the genus or the order increases. In %particular, one may ask:

\begin{question}\label{q:nilpotency}
Do there exist admissible braid quotients that are nilpotent of class at least $3$?
Are there non-nilpotent admissible braid quotients?
\end{question}

\section*{Acknowledgements}
Francesco Polizzi was partially supported by GNSAGA-INdAM.

%\section*{Research Data Policy and Data Availability Statements}
%Data sharing not applicable to this article as no datasets were generated
%or analysed during the %current study.

\end{document}